\documentclass[%draft,%
	11pt,oneside,british,a4paper]{amsart}
\usepackage[T1]{fontenc}
\usepackage[utf8]{inputenc}
\usepackage[dvipsnames]{xcolor}
\usepackage[numbers]{natbib}
\usepackage{amstext,amsthm,amssymb}
\usepackage{xparse,mathrsfs,mathtools}
\usepackage{bookmark}
\usepackage[capitalise,nameinlink]{cleveref}
\usepackage{babel}
\usepackage[yyyymmdd,hhmmss]{datetime}
\usepackage{hyperref}
\usepackage{lmodern,microtype}
\hypersetup{
	unicode=true,%
	pdfdisplaydoctitle=true,%
	pdfpagemode=UseOutlines,%
	breaklinks=true,%
	pdfencoding=unicode,psdextra,
	pdfcreator={},
	pdfborder={0 0 0},
	hidelinks
}

\newtheorem{thm}{Theorem}[section]
\newtheorem{prop}[thm]{Proposition}
\newtheorem{lem}[thm]{Lemma}
\newtheorem{cor}[thm]{Corollary}

\newcommand{\NN}{\mathbb{N}}
\newcommand{\ZZ}{\mathbb{Z}}
\newcommand{\QQ}{\mathbb{Q}}
\newcommand{\RR}{\mathbb{R}}
\newcommand{\PP}{\mathbb{P}}

\newcommand{\LandauO}{\mathcal{O}}

\newcommand{\sumTwoSquares}{\delta}

\newcommand{\funcName}{f}
\newcommand{\innerFunc}{g}

\newcommand{\ConstQAsympt}{\kappa}
\newcommand{\ConstFLemma}{c_*}
\newcommand{\ConstGDagger}{c_\dagger}
\newcommand{\ConstF}{c_f}
\newcommand{\ConstCe}{c_{\mathrm{e}}}

\newcommand{\numberfieldK}{K}
\newcommand{\integers}{\mathscr{O}}
\newcommand{\norm}{\mathrm{N}}

\DeclarePairedDelimiter\parentheses{\lparen}{\rparen}
\DeclarePairedDelimiter\braces{\lbrace}{\rbrace}
\DeclarePairedDelimiter\floor{\lfloor}{\rfloor}

\numberwithin{equation}{section}
\crefformat{equation}{(#2#1#3)}
\crefformat{enumi}{(#2#1#3)}

\title[The maximal order of iterated multiplicative functions]{%
	The maximal order of iterated multiplicative functions}
\date{\today}

\subjclass[2010]{11N37}
\keywords{maximal order, arithmetic function, iterated function}

\author{Christian~Elsholtz}
\address{Christian~Elsholtz\\%
	Institute of Analysis and Number Theory\\%
	Graz University of Technology\\%
	Kopernikusgasse~24\\%
	8010~Graz\\%
	Austria}
\email{elsholtz@math.tugraz.at}

\author{Marc~Technau}
\address{Marc~Technau\\%
	Institute of Analysis and Number Theory\\%
	Graz University of Technology\\%
	Kopernikusgasse~24\\%
	8010~Graz\\%
	Austria}
\email{mtechnau@math.tugraz.at}

\author{Niclas~Technau}
\address{Niclas~Technau\\
Raymond and Beverly Sackler School of Mathematical Sciences\\
Tel Aviv University\\
69978~Tel Aviv\\
Israel}
\email{niclast@mail.tau.ac.il}
\thanks{The third author was supported, while working on this paper, 
by the Austrian~Science~Fund~(FWF): by either Project
W1230~Doctoral~Program~``Discrete Mathematics'' 
or Project Y-901.}

\begin{document}
\begin{abstract}
	Following Wigert, various authors, including Ramanujan, Gronwall, Erd\H{o}s, Ivi\'{c}, Schwarz\nocite{SchwarzandWirsing:1973}, Wirsing, and Shiu, determined the maximal order of several multiplicative functions, generalizing Wigert's result
	\[
		\max_{n\leq x} \log d(n) = \frac{\log x}{\log \log x} (\log 2 + o(1)).
	\]
	
	On the contrary, for many multiplicative functions, 
	the maximal order of iterations of the functions
	remains widely open. The case of the iterated divisor function was only 
	solved recently, answering a question of Ramanujan from~1915.\\
	Here we determine the maximal order of $\log f(f(n))$ for a class of multiplicative functions $f$.
	In particular,
	this class contains	functions counting ideals of given norm
	in the ring of integers of an arbitrary, fixed quadratic number field.
	As a consequence, we determine such maximal orders for several multiplicative $f$ arising as a normalized function counting representations by certain binary quadratic forms.
	Incidentally, for the non-multiplicative function $r_2$ which counts how often a positive integer is represented as a sum of two squares, this entails the asymptotic formula
	\[
		\max_{n\leq x} \log r_2(r_2(n))= \frac{\sqrt{\log x}}{\log \log x} \parentheses{ c/\sqrt{2}+o(1) }
	\]
	with some explicitly given constant $c>0$.
\end{abstract}
\maketitle

% ~~~~~~~~~~~~~~~~~~~~~~~~~~~~~~~~~~~~~~~~~~~~~~~~~~~~~~~~~~~~~~~~~~~~
\section{Introduction}
\subsection{Maximal orders of multiplicative functions}
The study of the maximal order of arithmetic functions (for example of the
divisor functions $d$ or $\sigma$) is an integral part of introductory number
theory text books. For the divisor functions $d$ 
and $\sigma$, satisfactory answers are well known; see, 
for example, ~Wigert~\cite{Wigert:1907} 
and Gronwall~\cite{Gronwall:1913}.
% The methods make use of the fact that these are multiplicative functions. 
Their proofs make use of the fact that $d$ and $\sigma$ are multiplicative functions. 
For the maximal order of magnitude of \emph{iterated} arithmetic functions
much less is known. Here are some reasons which show that this is generally
a very delicate subject:
\begin{enumerate}
	\item The iterate of a multiplicative function need not be multiplicative; 
	for instance, for any pairwise distinct primes 
	$p_1,\ldots,p_r$,
	\[
		\frac{d(d(p_1))\cdots d(d(p_r))}{d(d(p_1\cdots p_r))}
		= \frac{d(2)^r}{d(2^r)} = \frac{2^r}{r+1} \neq 1.
	\]
	\item Let $a(n)$ denote the number of abelian groups of order $n$.
	By results of Erd\H{o}s and Ivi\'{c}~\cite{Erdos-Ivic:1989}
	it is known that
	\[ 
		\qquad \quad \exp \parentheses[\big]{ (\log x)^{1/2+o(1)} }
		\ll \max_{n \leq x} a(a(n))\ll \exp \parentheses[\big]{ ( \log x)^{7/8+o(1)} },
	\] 
	leaving a large gap between lower and upper bounds.
	Improving these bounds would seem to require understanding 
	the multiplicative structure of the number $p(n)$ of unrestricted
	partitions, about which very little is known beyond certain congruences.
	\item
	Let $\sigma_1(n)=\sigma(n)$ be the sum of divisors function, and
	$\sigma_k(n)=\sigma_1(\sigma_{k-1}(n))$ its iterates.
	Schinzel~\cite{Schinzel:1959} conjectured that 
	\[
		\liminf_{n\rightarrow \infty} \frac{\sigma_k(n)}{n} < \infty.
	\]
	This is only known for $k=1,2$ and $3$ by results of M\k{a}kowski
	\cite{Makowski:1976} and Maier~\cite{Maier:1984}, and
	conditionally on Schinzel's Hypothesis~H. 
	In light of studying maximal orders of magnitude: the equivalent question
	$\limsup_{n\rightarrow \infty} \frac{n}{\sigma_k(n)} > 0$ is equally open.
	\item For the iterated Euler $\varphi$-function the situation 
	is quite different, compared to the iterated $\sigma$-function.
	In view of $\varphi(2^n)=2^{n-1}$ and $\varphi_k(2^n)=2^{n-k}$, 
	it is evident that, for any fixed $k$, the extremal order of magnitude is
	$\lim \sup_{n \rightarrow\infty}  \frac{\varphi_k(n)}{n} 
	\geq \frac{1}{2^k}$.
	As Maier~\cite{Maier:1984} points out, the situation changes 
	if one discards such thin sets of prime powers.
	If one studies large values of $\varphi_k(n)$ that occur
	for about $\frac{x}{\log x}$ values of $n \leq x$, 
	then the situation is very similar to
	the situation with the iterated $\sigma$-function, see above.\\
	However, there are other non-trivial results
	on the iterated $\varphi$-function, e.g.\ concerning the range 
	of the values of $\varphi$ by Ford~\cite{Ford:1998} and $\varphi_k$
	(Luca and Pomerance~\cite{Luca-Pomerance:2009}).
	Moreover, it is well known that the iterated $\varphi$-function 
	has applications to Pratt-trees, see e.g.~\cite{Bayless:2008}.
\end{enumerate}

In the case of multiplicative functions, 
the maximal order of magnitude was initially proved in a number of 
individual cases:
the maximal order of the divisor function $d$ has 
been determined by Wigert~\cite{Wigert:1907} 
and Ramanujan~\cite{Ramanujan:1915}.
They proved that
\[
\limsup_{n \rightarrow \infty} \frac{\log d(n) \log \log n}{\log n}
=\log 2, 
\]
where $\log$ denotes the logarithm with base $e$.
(Note that for functions of this magnitude one typically 
has an asymptotic for $\log(f(n))$ rather than for $f(n)$ itself. From our
perspective we will still say that the maximal order has been determined.)
This study subsequently influenced 
(via results of Hardy and Ramanujan, Tur\'{a}n  and Erd\H{o}s and Kac)
the development of probabilistic number theory.

Ramanujan studied the multiplicative function $\sumTwoSquares$ that counts
the number of representations of its argument as a sum of two squares 
ignoring sign, i.e.,
\begin{align}
	\sumTwoSquares(n) = \tfrac{1}{4} \, \#\{ (x,y)\in\ZZ\times\ZZ 
	\mid x^2+y^2=n \}.
	\label{eq:r2:def}
\end{align}
If $\nu_p$ denotes the $p$-adic valuation, 
then it is well-known (see, e.g.,~\cite[Theorem~278]{hardy1954}) that
\begin{align}
	\sumTwoSquares(n) = \prod_{\substack{
		\text{prime } q \mid n \\ q\equiv 1\bmod 4
	}} (\nu_q(n)+1)
	\times \prod_{\substack{
		\text{prime } p \mid n \\ p\equiv 3\bmod 4
	}} \tfrac{1}{2} \parentheses[\big]{ 1+(-1)^{\nu_p(n)} }.
	\label{eq:r2:formula}
\end{align}
(To be precise, Ramanujan called this function $Q_2(n)$, but here we follow the notation used by Hardy and Wright~\cite[Theorem~278]{hardy1954}.)
We observe that $4 \sumTwoSquares(n) = r_2(n)$, where $r_2(n)$ is the sum of two squares function which also takes care of signs. 
Ramanujan~\cite{Ramanujan:1997} showed  that for some positive constant $a$
\[ \max_{n \leq x} \delta(n)= 
\exp\parentheses*{ \frac{\log 2}{2} \operatorname{li}(2 \log x) + \LandauO\parentheses[\big]{(\log x)\exp(-a \sqrt{\log
    x})} }, 
\]
where the right hand side can be simplified to
\[ \exp\parentheses*{ (\log 2 +o(1)) \frac{\log x}{\log \log x} }.\]
This implies the very same logarithmic maximum order:
\[\limsup_{n \rightarrow \infty} \frac{\log r_2(n) \log \log n}{\log n}
=\log 2. \]
Knopfmacher~\cite{Knopfmacher:1975} and Nicolas~\cite{Nicolas:1987},
who were unaware\footnote{%
	At that time Ramanujan's work was unpublished: 
	quite remarkably, the end of Ramanujan's paper~\cite{Ramanujan:1915} 
	of 1915 was not intended to be the end.
	In fact, Ramanujan's manuscript was considerably longer
	and due to a shortage of resources during wartime 
	the London Mathematical Society printed only part of the manuscript.
	The second part has been recovered and published many years later, 
	first in~\cite{Ramanujan:1988}, but later with detailed 
	annotations by Nicolas and Robin~\cite{Ramanujan:1997}, and also
	\cite{Andrews-Berndt:2012}.%
} of Ramanujan's work, later also observed this. 

Ramanujan (see~\cite[§§~55--56]{Ramanujan:1997})
also achieved the very same result
\[ \max_{n \leq x} \bar{Q}_2(n)= 
\exp \parentheses*{ \frac{\log 2}{2} \operatorname{li}(2 \log x) + \LandauO\parentheses[\big]{(\log x)\exp(-a \sqrt{\log
    x})} }, 
\]
for the function
$\bar{Q}_2(n)$ counting non-negative pairs $(x,y)$ with $n=x^2+xy+y^2$,
\begin{equation}\label{eq:RamanujanQbar2}
	\bar{Q}_2(n) = \prod_{\substack{
		\text{prime } q \mid n \\ q\equiv 1\bmod 3
	}} (\nu_q(n)+1)
	\times \prod_{\substack{
		\text{prime } p \mid n \\ p\equiv 2\bmod 3
	}} \tfrac{1}{2} \parentheses[\big]{ 1+(-1)^{\nu_p(n)} }.
\end{equation}
Note that this quadratic form corresponds to the Eisenstein lattice 
$\mathbb{Z}[e^{2\pi i/3}]$, and non-negative coordinates correspond 
to a sector of $60$ degrees, 
which explains the factor $1/6$ in~\cref{eq:QForm:Eisenstein}; 
for a more conceptual 
explanation for the factor $1/6$, 
see the last display formula before \cref{sec:results}.

Kr\"atzel~\cite{Kraetzel:1970}
proved for the number $a(n)$ of non-isomorphic abelian groups of 
order $n$:
\[
	\limsup_{n \rightarrow \infty} \frac{\log a(n) \log \log n}{\log n}
	=\frac{1}{4} \log 5, 
\]
and Knopfmacher~\cite{Knopfmacher:1973} proved 
for the number $\beta(n)$ of squareful divisors of $n$:
\[
	\limsup_{n \rightarrow \infty} \frac{\log \beta(n) \log \log n}{\log n}
	=\frac{1}{3} \log 3. 
\]
Note that all of the functions $a$, $\beta$ and $d$ are prime independent, where a multiplicative arithmetic function $f$ is said to be \label{notion:PrimeIndependent}\emph{prime independent} if $f(p^\nu) = f(2^\nu)$ for every prime power $p^\nu$.

A number of authors independently observed that
such limits can be worked out more generally for the class of prime independent multiplicative functions. 
Of these results we only mention the one by Shiu~\cite{Shiu:1980},
but there are others---see~\cite{Babanazarov-Podzarskii:1987,DrozdovaandFreiman:1958,Heppner:1973,Hilberdink:2014,Knopfmacher:1975,Nicolas:1980,Norton:1992,Postnikov:1988,Suryanarayana-Rao:1975}.

Shiu~\cite{Shiu:1980} proved:
let $\funcName:\NN\rightarrow \RR$ be a 
multiplicative function satisfying the following conditions:
\begin{enumerate}
	\item There exist constants $A$ and $0 < \theta <1$ such that
	$\funcName(2^\nu) \leq \exp(A \nu^{\theta})$ where $\nu\geq 1$, and 
	\item for all primes $p$ and all $a\geq 1$ one has
	$\funcName(p^\nu)=\funcName(2^\nu)\geq 1$, then the following holds:
	\[\limsup_{n \rightarrow \infty}
	\frac{\log \funcName(n) \log \log n}{\log n}=\log \max_{\nu\geq 1} (\funcName(2^\nu))^{1/\nu}. \]
\end{enumerate} 

\subsection{On iterates of arithmetic functions}

The quest for the maximal order of the iterated divisor function 
was raised by Ramanujan~\cite{Ramanujan:1915}
in his paper on highly composite numbers.
At the very end of that paper he gave a construction of integers, namely,
$N_k=\prod_{i=1}^k p_i^{p_i-1}$, where $p_i$ denotes the $i$-th prime, and observed that for these integers
$d(d(N_k))\geq 
\exp \parentheses[\big]{ \parentheses{ \sqrt 2 \log 4 +o(1) }
\frac{\sqrt{ \log N_k}}{\log \log N_k} }$ 
holds.
Erd\H{o}s and K\'{a}tai~\cite{ErdoesandKatai:1969}, 
Ivi\'{c}~\cite{Ivic:1995} and Smati 
\cite{Smati:2005, Smati:2008}
gave results on the maximal order, but a satisfying answer
on the maximal order of the iterated divisor function 
was only given almost 100 years after Ramanujan's paper: 
Buttkewitz, Elsholtz, Ford and Schlage-Puchta 
\cite{ButtkewitzElsholtzFord-SP:2012} proved:
\[\limsup_{n \rightarrow \infty}
\frac{\log d(d(n)) \log \log n}{\sqrt{\log n}}=c, \]
where
\begin{equation}\label{eq:B--E--F--S-P:constant}
	c = \parentheses*{ 
		8\sum_{l=1}^{\infty} 
			\parentheses*{ \log \parentheses*{ 1+\frac{1}{l} } }^{2}
	}^{1/2}
%	= 2.79598\ldots\,
	.
\end{equation}
However, it seems that no similar result is known for either of the functions $\delta$ or $r_2$. Neither of these functions is prime independent in the sense defined above, but nonetheless they are still quite similar to $d$ (compare~\cref{eq:r2:formula}). For the latter reason, results concerning $\delta$ and $r_2$ have often been an intuitive next step following results concerning $d$.
In fact, let us recall the development for sums
of multiplicative functions, where Landau investigated the number of integers
representable as sums of two squares. Subsequently, this was
generalized many times, for example to the number of integers 
consisting of primes in certain residue classes only, and eventually led
to the celebrated mean value results of Wirsing and Hal\'{a}sz.

Motivated by this development, we study a class of multiplicative functions
which includes important functions, such as the divisor function $d$, $\delta$, and---more generally---a number of functions connected with counting ideals in quadratic number fields.
In the spirit of Shiu's theorem, we also investigate which hypotheses on the function $\funcName$
and which growth rates of $\funcName(p^\nu)$, depending on $\nu$, 
allow us to bound the maximum order magnitude of $\funcName(\funcName(n))$.
In some cases (including $\delta$, $r_2$ and $\bar{Q}_2$),
we are able to give an asymptotic for the logarithmic size of this maximum.

\subsection{Plan of the paper}
The rest of the paper is structured as follows: first, we present our results in ascending generality.
Results for $\delta$, $r_2$, $\bar{Q}_2$, and some related functions are presented in \cref{sec:delta:and:r2}.
In \cref{sec:AlgNT}, these results are then cast into a more conceptual light from the point of view of basic algebraic number theory.
All of these results follow from general results we describe in \cref{sec:results}.
The rest of the paper deals with supplying all the deferred proofs (which mostly concerns our statements from \cref{sec:results}).

% ~~~~~~~~~~~~~~~~~~~~~~~~~~~~~~~~~~~~~~~~~~~~~~~~~~~~~~~~~~~~~~~~~~~~
\section{The prototypes: $\delta$, \texorpdfstring{$r_2$}{r₂}, \texorpdfstring{$\bar{Q}_2$}{Q̅₂}, and relatives}\label{sec:delta:and:r2}

The following result is an immediate consequence of~\cref{eq:r2:formula} and \cref{cor:DivisorLike:MaxOrder} below.
\begin{thm}\label{cor:g2:MaxOrder}
	Let $\sumTwoSquares$ be given by~\cref{eq:r2:def}. Then
	\[
	\max_{n\leq x} \log \sumTwoSquares(\sumTwoSquares(n))
	= \frac{\sqrt{\log x}}{\log_2 x} \parentheses*{
		\frac{c}{\sqrt{2}}
		+ \LandauO\parentheses*{\frac{\log_3 x}{\log_2 x}}
	},
	\]
	where $c$ is given in~\cref{eq:B--E--F--S-P:constant}.
\end{thm}
Incidentally, this implies a result for $r_2$ 
at no additional effort---even though $r_2$ 
is \emph{not} multiplicative:
\begin{cor}\label{cor:MaxOrder:r2}
	The assertion of \cref{cor:g2:MaxOrder} remains valid if $\delta$ is replaced with $r_2$, where $r_2(n) = 4\delta(n) = \#\{ (x,y)\in\ZZ\times\ZZ 
	\mid x^2+y^2=n \}$.
\end{cor}
\begin{proof}
	For $n\in\NN$ write $\delta(n) = 2^\nu m(n)$ with some odd integer $m(n)$. Then, using multiplicativity of $\delta$ and $\delta(2^{2+\nu}) = 1 = \delta(2^{\nu})$, we have
	\begin{align*}
		r_2(r_2(n)) &
		= 4\delta(2^{2+\nu} m(n))
		= 4\delta(2^{2+\nu})\delta(m(n)) \\ &
		= 4\delta(2^{\nu})\delta(m(n))
		= 4\delta(2^{\nu}m(n))
		= 4\delta(\delta(n)).
	\end{align*}
	Hence, $\log r_2(r_2(n)) = \log \delta(\delta(n)) + 2\log 2$ and the assertion of the corollary follows from \cref{cor:g2:MaxOrder}.
\end{proof}

It turns out that our arguments are not just limited 
to the binary quadratic form $x^2+y^2$ appearing 
in~\cref{eq:r2:formula}. A more refined explanation 
can be found in the next section. Here we content 
ourselves with stating the next result in a very modest form:
\begin{thm}\label{thm:forms}
	Fix $k\in\braces{2,3,5,11,17,41}$ and let $f(n)$ be defined by either of the following expressions
	\begin{gather}
		\label{eq:QForm:Gaussian}
		\tfrac{1}{4} \#\braces{ (x,y)\in\ZZ^2 \mid x^2 + y^2 = n }, \\
		\label{eq:QForm:SqrtMinusTwo}
		\tfrac{1}{2} \#\braces{ (x,y)\in\ZZ^2 \mid x^2 + 2 y^2 = n }, \\
		\label{eq:QForm:Eisenstein}
		\tfrac{1}{6} \#\braces{ (x,y)\in\ZZ^2 \mid x^2 + x y + y^2 = n }, \\
		\notag
		\tfrac{1}{2} \#\braces{ (x,y)\in\ZZ^2 \mid x^2 + x y + k y^2 = n }
	\end{gather}
	for all positive integers $n$ and put $f(0)=1$.
	Then
	\[
	\max_{n\leq x} \log f(f(n))
	= \frac{\sqrt{\log x}}{\log_2 x} \parentheses*{
		\frac{c}{\sqrt{2}}
		+ \LandauO\parentheses*{\frac{\log_3 x}{\log_2 x}}
	},
	\]
	where $c$ is given in~\cref{eq:B--E--F--S-P:constant}.
	% The unique(!) reduced binary quadratic form (A,B,C) <-> Ax^2+Bxy+Cy^2 with discriminant D and class number h(D)=1:
	%     D    (A,B,C)    behaviour of (2)
	%                        in the ring of integers of\QQ(\sqrt{D})
	% ----------------------------------------------------------------
	%    -3    (1,1,1)    inert
	%    -4    (1,0,1)    RAMIFIED
	%    -7    (1,1,2)    split
	%    -8    (1,0,2)    RAMIFIED
	%   -11    (1,1,3)    inert
	%   -19    (1,1,5)    inert
	%   -43    (1,1,11)   inert
	%   -67    (1,1,17)   inert
	%  -163    (1,1,41)   inert
\end{thm}
Certainly \cref{thm:forms} implies \cref{cor:g2:MaxOrder} (see~\cref{eq:QForm:Gaussian}).
Moreover, it covers the choice $f=\bar{Q}_2$ with $\bar{Q}_2$ defined in~\cref{eq:RamanujanQbar2} (see~\cref{eq:QForm:Eisenstein}).
\cref{cor:MaxOrder:r2} allows one to drop the factor~$\tfrac{1}{4}$ in~\cref{eq:QForm:Gaussian} and the same trick used for proving \cref{cor:MaxOrder:r2} can also be used to show that one can drop the factor~$\tfrac{1}{2}$ in~\cref{eq:QForm:SqrtMinusTwo}.

% ~~~~~~~~~~~~~~~~~~~~~~~~~~~~~~~~~~~~~~~~~~~~~~~~~~~~~~~~~~~~~~~~~~~~
\section{Examples from algebraic number theory}\label{sec:AlgNT}
Our next objective is to fit \cref{thm:forms} into a broader context. We begin with some notation.
Let $\numberfieldK$ be an arbitrary number field and $\integers$ its ring of algebraic integers. The \emph{norm} of an ideal $\mathfrak{a}\subseteq\integers$ is denoted by $\norm\mathfrak{a} = \#(\integers/\mathfrak{a}\integers) \in \NN$. Consider
\begin{equation}\label{eq:f_{\numberfieldK}:def}
	f_{\numberfieldK}(n) = \#\braces{ \text{ideals }\mathfrak{a}\subseteq\integers : \norm\mathfrak{a} = n }.
\end{equation}
For two ideals $\mathfrak{a},\mathfrak{b}\subseteq\integers$ their product $\mathfrak{ab}$ is defined element-wise and we have $\norm(\mathfrak{ab}) = (\norm\mathfrak{a}) (\norm\mathfrak{b})$. On combining this with the classical facts that ideals of $\integers$ admit a unique factorisation into prime ideals and that the norm of a prime ideal is always a power of a prime in $\mathbb{N}$, one easily deduces that $f_{\numberfieldK}:\NN_0\to\NN_0$ as defined above is multiplicative.\\
In order to understand the behaviour of $f_{\numberfieldK}$ on prime powers $p^\nu$, we start by considering any ideal $\mathfrak{a}$ with $\norm\mathfrak{a} = p^\nu$. By the multiplicativity of the norm, the factorisation of $\mathfrak{a}$ consists only of prime ideals whose norm is again a power of $p$. From algebraic number theory one knows that the number of prime ideals $\mathfrak{p}\subseteq\integers$ with $p$ dividing $\norm\mathfrak{p}$ is finite. Thus, enumerating the prime ideals with said property by $\mathfrak{p}_1,\ldots,\mathfrak{p}_k$, we may write $\mathfrak{a} = \mathfrak{p}_1^{\nu_1}\cdots \mathfrak{p}_k^{\nu_k}$ for some exponents $\nu_1,\ldots,\nu_k\in\NN_0$. The condition that $\norm\mathfrak{a} = n$ then takes the form
\begin{equation}\label{eq:IdealFactorisation}
	\sum_{j=1}^k \nu_j
	%	\underbrace{
	\frac{\log\norm\mathfrak{p}_j}{\log p}
	%	}_{\in\mathbb{N}}
	= \nu.
\end{equation}
On the other hand, this argument also works in the opposite direction. 
Hence, the map
\begin{equation}\label{eq:ExponentCorrespondence}
	\begin{aligned}
		\braces{ (\nu_1,\ldots,\nu_k)\in\NN_0^k : \text{\cref{eq:IdealFactorisation} holds} } &
		\longrightarrow
		\braces{\text{ideals \(\mathfrak{a}\subseteq\integers\) with \(\norm\mathfrak{a} = p^\nu\)}}, \\
		(\nu_1,\ldots,\nu_k) &
		\longmapsto \mathfrak{p}_1^{\nu_1}\cdots \mathfrak{p}_k^{\nu_k}
	\end{aligned}
\end{equation}
is bijective. Moreover, one knows that the principal ideal $(p) = p\integers$ factors as $(p) = \mathfrak{p}_1^{e_1}\cdots \mathfrak{p}_k^{e_k}$, so that
\begin{equation}\label{eq:FundamentalFormula}
	\sum_{j=1}^k e_j \frac{\log\norm\mathfrak{p}_j}{\log p} = \frac{\norm(p)}{\log p} = [\numberfieldK:\QQ],
\end{equation}
where the right hand side is the 
degree of the field extension $\numberfieldK/\QQ$.

If $[\numberfieldK:\QQ]>2$, then~\cref{eq:IdealFactorisation} allows for too much freedom in the choice of the exponents $\nu_1,\ldots,\nu_k$. Consequently, the value of $f_K$ at prime powers $p^\nu$ may depend too loosely on the prime $p$ and the quick growth of $f_K(p^\nu)$ as a function of $\nu$ poses additional problems. The situation becomes appreciably better if $\numberfieldK$ is a quadratic extension of $\QQ$ and we shall henceforth restrict ourselves to this case.
Then, by~\cref{eq:FundamentalFormula}, for any prime $p$, only one of the following three cases may occur:
\begin{enumerate}
	\item $p$ is \emph{ramified}, that is, $(p) = \mathfrak{p}_1^2$ for some prime ideal $\mathfrak{p}_1$ of $\integers$ with $\norm\mathfrak{p}=p$;
	\item $p$ is \emph{split}, that is, $(p) = \mathfrak{p}_1 \mathfrak{p}_2$ with some distinct prime ideals $\mathfrak{p}_1,\mathfrak{p}_2\subseteq\integers$ each having norm $p$;
	\item $p$ is \emph{inert}, that is, $(p) = \mathfrak{p}_1$ is a prime ideal of $\integers$ and $\norm\mathfrak{p}_1 = \norm(p) = p^2$.
\end{enumerate}
>From multiplicativity of $f_{\numberfieldK}$ in combination with the map in~\cref{eq:ExponentCorrespondence} being bijective, it follows that
\begin{equation}\label{eq:f_K:representation}
	f_{\numberfieldK}(n)
	= \prod_{\substack{\text{prime }p\mid n\\ p \text{ ramified}}} 1
	\times \prod_{\substack{\text{prime }q\mid n\\ q \text{ split}}} (\nu_q(n)+1)
	\times \prod_{\substack{\text{prime }p\mid n\\ p \text{ inert}}} \tfrac{1}{2} \parentheses[\big]{ 1+(-1)^{\nu_p(n)} }.
\end{equation}
>From this representation of $f_K(n)$, we see that our results from \cref{sec:results} below can be applied to establish the following:
\begin{thm}\label{thm:g2:MaxOrder:numberfield}
	Suppose that $\numberfieldK$ is some fixed quadratic number field and let $f_{\numberfieldK}$ be given by~\cref{eq:f_{\numberfieldK}:def}. Then
	\[
		\max_{n\leq x} \log f_{\numberfieldK}(f_{\numberfieldK}(n))
		= \frac{\sqrt{\log x}}{\log_2 x} \parentheses*{
			\frac{c}{\sqrt{2}}
			+ \LandauO\parentheses*{\frac{\log_3 x}{\log_2 x}}
		},
	\]
	where the implied constant may depend on $\numberfieldK$, and $c$ is given in~\cref{eq:B--E--F--S-P:constant}.
\end{thm}
The missing pieces for the proof of the above theorem are given in \cref{sec:thm:g2:MaxOrder:numberfield:proof} below.

In the remainder of this section we briefly sketch how to obtain \cref{thm:forms} from \cref{thm:g2:MaxOrder:numberfield}. To this end, assume that $\numberfieldK$ is some imaginary quadratic number field with class number one. Then the value $f_{\numberfieldK}(n)$ can be viewed as the number of solutions to some binary quadratic equation. Indeed, the assumption about the class number implies that all ideals of $\integers$ are principal and $\numberfieldK$ being imaginary quadratic implies that $\integers$ has finitely many units. Therefore,
\[
f_{\numberfieldK}(n) = \frac{\#\braces{\xi\in\integers : \norm(\xi) = n}}{\#\braces{\text{units in }\integers}}.
\]
The celebrated \emph{Baker--Heegner--Stark theorem} gives a complete classification (up to isomorphism) of all imaginary quadratic 
number fields with class number one.
\Cref{thm:forms} then follows immediately by going through that list, rewriting the norm equation $\norm(\xi)=n$ with respect to some integral basis and applying \cref{thm:g2:MaxOrder:numberfield}; here the choice of the integral basis may affect the particular form one gets, but any such form can be readily checked to be equivalent to one of the ones implicit in \cref{thm:forms} by using the well-known reduction theory for binary quadratic forms.

% ~~~~~~~~~~~~~~~~~~~~~~~~~~~~~~~~~~~~~~~~~~~~~~~~~~~~~~~~~~~~~~~~~~~~
\section{Hypotheses and the general results}\label{sec:results}
In what follows, we give a description of a class of arithmetic functions
for which the subsequent reasoning works. The imposed restrictions
could be relaxed somewhat, but the model cases we primarily aim at are given in~\cref{eq:r2:formula},~\cref{eq:RamanujanQbar2} and~\cref{eq:f_K:representation}. The important features here are the following:
the arithmetic function $f$ to be iterated is multiplicative,
acts affinely on the exponents of powers of primes $q$ from a
certain subset of primes $Q\subseteq\PP$ (e.g., $\text{primes}\equiv1\bmod4$ in the case $f=\delta$ as seen in~\cref{eq:r2:formula}),
and takes only the values $0,1$ on powers of primes $p\in\PP\setminus Q$
(subject to a rule which---under the assumptions below---turns out to irrelevant).

In~\cite{ButtkewitzElsholtzFord-SP:2012}, the case
$Q=\PP$ with the multiplicative
arithmetic function $d$ acting as $d(p^\nu)=\nu+1$ is studied.
In our approach, we assume that $f(q^\nu)=g(\nu)$ for powers of primes 
$q\in Q$ with a function $g$ satisfying suitable axioms as listed below.
By elaborating on the method of~\cite{ButtkewitzElsholtzFord-SP:2012}, we obtain upper and lower bounds on the
maximal order of first iterates of arithmetic functions $f$ 
which enjoy similar properties as those observed for $d$ 
and $\sumTwoSquares$, see \cref{thm:UpperBound} and \cref{thm:LowerBound}.

In detail, we start with a strictly increasing sequence of primes
$(q_{j})_{j\geq1}$.
By the prime number theorem, the sequence of all primes $(p_{j})_{j\geq1}$ satisfies $p_j=j(\log j+\log (\log j)+\LandauO(1))$, so it seems reasonable to assume similar asymptotics for $(q_{j})_{j\geq1}$ (see~\cref{assumption:qAsymptotic} below).
Set $Q=\{ q_j : j\in\NN \}$ and let
$\langle Q\rangle$ be the monoid (multiplicatively) generated by $Q$.
Furthermore, fix a map $\innerFunc:\NN_0\to\NN$ with $\innerFunc(0)=1$
and let\footnote{The symbol $\innerFunc^{\dagger}$ 
was chosen to allude to a pseudo inverse.}
\begin{equation}
	\innerFunc^{\dagger}(y) = \inf\{ x\in\NN : \innerFunc(x)=y \} \in \NN\cup\braces{+\infty}.
	\label{eq:def:gDagger}
\end{equation}
Finally, assume that
\begin{enumerate}
	\renewcommand{\theenumi}{A.\arabic{enumi}}
	% ----------------------------------------------------------------
	\item \label{assumption:qAsymptotic}
	$(q_{j})_{j\geq1}$ satisfies the asymptotic expansion
	\[
		q_j = \ConstQAsympt j \parentheses{ \log j+\log (\log j) + \LandauO(1) },
	\]
	where $\ConstQAsympt>0$ is some constant,
	% ----------------------------------------------------------------
	\item% \label{assumption:fMonotone}
	$\innerFunc$ is monotonically increasing,
	% ----------------------------------------------------------------
	\item \label{assumption:fValues}
	% (This ensures that $\innerFunc^\dagger$ 
	%is defined on $\langle Q\rangle$.)
	$\innerFunc(\NN) \supseteq \langle Q\rangle$,
	% ----------------------------------------------------------------
	\item \label{assumption:fLemma}
	$\innerFunc^{\dagger}(b)+\ConstFLemma b \innerFunc^{\dagger}(a) 
	\leq \innerFunc^{\dagger}(ab)$
	for all $a,b\in\langle Q\rangle$ such that $q_1\leq a\leq b$,
	where $\ConstFLemma > 1/q_1$ is some constant,
	% ----------------------------------------------------------------
	\item \label{assumption:fGrowth}
	$\innerFunc(i)/\innerFunc(i-1) = 1 + \LandauO(i^{-1/2-\epsilon})$
	for some $\epsilon>0$,
	% ----------------------------------------------------------------
	\item \label{assumption:fsublinear}
	$\innerFunc(x)\leq\ConstF x$ for all $x\in\NN$,
	where $\ConstF>0$ is some constant,
	% ----------------------------------------------------------------
	\item \label{assumption:gDaggerGrowth}
	$\innerFunc^{\dagger}(q) = \ConstGDagger q + \LandauO(q/\log q)$ 
	as $Q\ni q\to\infty$, where $\ConstGDagger>0$ is some constant. 
	(Note that $\innerFunc^{\dagger}(q)$ 
	is finite due to~\cref{assumption:fValues}.)
	% ----------------------------------------------------------------
	\end{enumerate}

Now let $\funcName$ be a multiplicative arithmetic function satisfying
\begin{align}
	\funcName(p^{\nu})\begin{cases}
		= \innerFunc(\nu) & \text{if } p\in Q, \\
		\in \{0,1\}	& \text{if } p\notin Q
	\end{cases}
	\label{eq:f:def}
\end{align} for a prime power $p^\nu\geq 1$. 
Furthermore, let $\funcName(0)=1$. We write
\begin{equation}
	M(x) = \max_{n\leq x} \log \funcName(\funcName(n)).
	\label{eq:M:def}
\end{equation}
On writing $\log_k$ for the $k$-fold iterate of the natural logarithm, 
%the zero-th iterate being the $\logfunction itself,
our main results may now be stated as follows:
\begin{thm}\label{thm:UpperBound}
	Let $M$ be as in~\cref{eq:M:def}.
	Then,
	\begin{align}
		M(x) \leq \frac{\sqrt{\log x}}{\log_2 x} \parentheses*{
			\frac{C_g}{\sqrt{\ConstQAsympt\ConstGDagger}}
			+ \LandauO\parentheses*{\frac{\log_3 x}{\log_2 x}}
		},
		\label{eq:thm:UpperBound}
	\end{align}
	where the implied constant depends on $Q,f$ and
	\begin{align}
		C_g = \parentheses[\Bigg]{
			8\sum_{j=1}^{\infty} \parentheses*{
			\log\frac{\innerFunc(j)}{\innerFunc(j-1)}}^2
		}^{1/2}.
		\label{eq:def:C}
	\end{align}
\end{thm}
Throughout this paper, $C_g$ always denotes
the constant defined in~\cref{eq:def:C}.
We also note in passing that throughout all implied constants 
may depend on the function $\funcName$ and the set $Q$ 
and an $\epsilon$, where obvious.
\begin{thm}\label{thm:LowerBound}
	Letting $\innerFunc(\nu) = \alpha\nu + 1$ and assuming the above hypotheses,
	the following holds
	\begin{align}
		M(x) \geq \frac{\sqrt{\log x}}{\log_2 x} \parentheses*{
			\frac{C_g}{\sqrt{\ConstQAsympt/\alpha}}
			+ \LandauO\parentheses*{\frac{\log_3 x}{\log_2 x}}
		}.
		\label{eq:thm:LowerBound}
	\end{align}
\end{thm}

Upon combining \cref{thm:UpperBound} 
and \cref{thm:LowerBound}, 
we immediately deduce the following corollary:
\begin{cor}\label{cor:DivisorLike:MaxOrder}
	Letting $\innerFunc(\nu) = \alpha\nu + 1$
	for some $\alpha\in\NN$, and on the above hypotheses, it holds that
	\[
	M(x)
	= \frac{\sqrt{\log x}}{\log_2 x} \parentheses*{
		\frac{C_g}{\sqrt{\ConstQAsympt/\alpha}}
		+ \LandauO\parentheses*{\frac{\log_3 x}{\log_2 x}}
	}.
	\]
\end{cor}

We note in passing that, for $Q=\PP$ in the setting of \cref{cor:DivisorLike:MaxOrder}, the function $\funcName$ in~\cref{eq:f:def} arises naturally as number of divisors of monic monomials, i.e., $\funcName(n) = d(n^\alpha)$.

% ~~~~~~~~~~~~~~~~~~~~~~~~~~~~~~~~~~~~~~~~~~~~~~~~~~~~~~~~~~~~~~~~~~~~
\section{Proof of \texorpdfstring{\cref{thm:g2:MaxOrder:numberfield}}{Theorem\autoref{thm:g2:MaxOrder:numberfield}}}\label{sec:thm:g2:MaxOrder:numberfield:proof}
We assume the notation of \cref{thm:g2:MaxOrder:numberfield} and recall~\cref{eq:f_K:representation}.
In order to use \cref{cor:DivisorLike:MaxOrder} to obtain an asymptotic formula for
\[
	\max_{n\leq x} \log f_{\numberfieldK}(f_{\numberfieldK}(n)),
\]
it only remains to verify Assumption~\cref{assumption:qAsymptotic}, where $q_j$ therein is taken to be the $j$-th smallest prime which splits in $\numberfieldK$.
The next lemma furnishes a prime number theorem for such primes and can be used to verify that Assumption~\cref{assumption:qAsymptotic} holds with $\ConstQAsympt=\frac{1}{2}$ and thereby finishes the proof of \cref{thm:g2:MaxOrder:numberfield}.
\begin{lem}\label{lem:SplittingAsymptotics}
	Suppose that $\numberfieldK$ is some fixed quadratic number field. Then
	\[
	\#\braces{ \text{primes } q\leq x \text{ which are split in } \numberfieldK } = \frac{1}{2} \frac{x}{\log x}(1+\LandauO(1/\log x))
	\]
	as $x\to\infty$.
\end{lem}
The above lemma is certainly well-known. Nevertheless, we sketch a proof for the convenience of the reader.
\begin{proof}[Proof of \cref{lem:SplittingAsymptotics}]
	By~\cite[Proposition~13.1.3]{Ireland-Rosen:1990} there is some integer $\Delta$ such that
	\begin{align*}
	\MoveEqLeft
	\#\braces{ \text{primes } q\leq x \text{ which are split in } \numberfieldK } \\ &
	= \#\braces*{ \text{primes } q\leq x \text{ with } \parentheses*{\frac{\Delta}{q}}=1 } + \LandauO(d(\Delta)) \\ &
	= \frac{1}{2} \#\braces{ \text{primes } p\leq x} + \frac{1}{2}\sum_{\text{primes }q\leq x} \parentheses*{\frac{\Delta}{q}} + \LandauO(d(\Delta))
	\end{align*}
	where $\parentheses[\big]{\tfrac{\Delta}{q}}$ denotes the Kronecker symbol. Consequently, the assertion of the lemma then follows from the prime number theorem and~\cite[Corollary~5.29]{Iwaniec-Kowalski:2004} (see also~\cite[Exercise~6 in~§3.5]{Iwaniec-Kowalski:2004}).
\end{proof}

% ~~~~~~~~~~~~~~~~~~~~~~~~~~~~~~~~~~~~~~~~~~~~~~~~~~~~~~~~~~~~~~~~~~~~
\section{Notation and auxiliary results}%\label{sec:preliminaries}
\subsection{Notation}
At this point it is convenient to introduce some
additional notation used throughout the rest of the paper.
Let
\begin{itemize}
	\item $\Omega(n)=\sum_{p\mid n} \nu_p(n)$, $\omega(n)=\sum_{p\mid n} 1$,
	\item $
		\Pi_{Q}(n)
		= \max\{ m\in\langle Q\rangle : m\mid n \}
	$,
	\item $\Omega_{Q} = \Omega\circ\Pi_{Q}$
	\item $\omega_{Q} = \omega\circ\Pi_{Q}$,
	\item $\pi_{Q}(x) = \#\{ q\in Q : q\leq x \}$.
\end{itemize}

\subsection{Auxilliary results}
We would like to give the reader our perspective on the problem at hand.
In order to keep the notation simple, let $Q=\PP$ for the moment.
Then, for any positive integer $n$,
\[
	\log \funcName(\funcName(n)) =
	\sum_{\substack{ q\in Q \\ q \mid \funcName(n) }}
	\log \innerFunc \parentheses[\big]{ q^{\nu_q(\funcName(n))} }.
\]
Vaguely speaking, in order to give estimates on $M(x)$, one needs
to exhibit some control over the prime factors of integers $N$,
which appear as values $N=\funcName(n)$ for $n\leq x$.
This sort of control is provided by \cref{lem:m_N}.

Additionally, one might like to remove~$\innerFunc$ from the above sum and
perhaps also take advantage of the fact that (weighted) sums of
$\nu_q(\funcName(N))$ over~$q$ are more readily controlled than values of
$\nu_q(\funcName(N))$ for some individual~$q$. \cref{lem:Cauchy-Schwarz}
makes this happen and is the source of the main term in
\cref{thm:UpperBound} and \cref{thm:LowerBound}.

Finally, \cref{lem: upper estimate for maximal order of f}
is a technical tool used to handle the case when $N=\funcName(n)$ does not
have sufficiently many prime factors~$q$ with small exponent~$\nu_q(N)$.

\begin{lem}\label{lem:m_N}
	For an $N \in \langle Q\rangle$, 
	let $m_{N}$ be the least positive integer $m$
	such that $\funcName(m)=N>1$. Then the following assertions hold:
	\begin{enumerate}
		\item \label{lem:m_N:factorisation}
		The number $m_N$ factors as $m_N = q_1^{\nu_1} \cdots q_r^{\nu_r}$ with some $r\geq 1$, exponents
		$\nu_1 \geq \ldots \geq \nu_r$ and the primes $q_1,q_2,\ldots$ from \cref{sec:results}.
		\item \label{lem:m_N:division}
		If $N'$ divides $N$, then $m_{N'} \leq m_N$.
		\item \label{lem:m_N:OmegaEstimate}
		If $q_j > q_{r+1}^{1/s_k}$ for some
		$j \leq r$, then $\Omega(\innerFunc(\nu_j)) \leq k$,
		where $s_k = \ConstFLemma q_1^k$.
	\end{enumerate}
\end{lem}
\begin{proof}
	Pick some $p\notin Q$ and let $\nu=\nu_p(m_N)$. Then 
	$1<N=\funcName(m_N)=\funcName(p^\nu) \funcName(m_N/p^\nu)$, so
	that $m_N=m_N/p^\nu$. Hence, $\nu=0$ and $p\nmid m_N$.
	Now, writing $m_N = q_1^{\nu_1} \cdots q_r^{\nu_r}$,
	note that one can permute the exponents without changing the
	value under $\funcName$. Therefore, by minimality of $m_N$, we must
	have $\nu_1 \geq \ldots \geq \nu_r$.
	This proves~\cref{lem:m_N:factorisation}.
	
	% ----------------------------------------------------------------
	Turning to~\cref{lem:m_N:division}, if we write 
	$m_N = q_1^{\nu_{1}}\cdots q_r^{\nu_r}$,
	then, by~\cref{eq:f:def},
	\begin{equation}\label{eq:N:g:factorisation}
		N = f(m_N) = \prod_{j\leq r} \innerFunc(\nu_j),
	\end{equation}
	and since $N'\mid N$ there is a partition
	$\nu_k = \nu_{k,1} + \ldots + \nu_{k,r}$
	such that
	$N_j' = \prod_{k\leq s} q_{k}^{\nu_{k,j}} \mid \innerFunc(\nu_j)$.
	By Assumption~\cref{assumption:fValues} on $\innerFunc$,
	the value $N_j'$ is attained by $\innerFunc$.
	Hence, we may look at $m_* = q_{1}^{\nu_1'}\cdots q_{r}^{\nu_r'}$,
	where $\nu_j' = \innerFunc^{\dagger}(N_j')$.
	Clearly, $\funcName(m_*)=N'$, and, by monotonicity of $\innerFunc$,
	$\nu_j'\leq\nu_{j}$, so that
	$m_{N'} \leq m_* \leq m_N$.
	
	% ----------------------------------------------------------------
	To prove~\cref{lem:m_N:OmegaEstimate}, let us assume for the
	sake of contradiction that $q_{j}>q_{r+1}^{1/s_{k}}$,
	$\Omega(\innerFunc(\nu_{j}))>k$.
	Using~\cref{eq:N:g:factorisation} and $N\in\langle Q\rangle$, we have $\Omega_{Q}(\innerFunc(\nu_{j})) = \Omega(\innerFunc(\nu_{j})) > k$, so that there is a decomposition $\innerFunc(\nu_{j}) = ab$, where
	$a\geq q_1$, $b\geq q_1^{k}$.
	Recalling the definition of $\innerFunc^{\dagger}$ in~\cref{eq:def:gDagger} and using that both $a$ and $b$ are contained in $\langle Q\rangle$, we see that both $\innerFunc^{\dagger}(a)$ and $\innerFunc^{\dagger}(b)$ are finite and we may consider
	\[
		m^*
		= q_j^{\innerFunc^{\dagger}(b)} q_{r+1}^{\innerFunc^{\dagger}(a)}
		\prod_{\substack{i=1 \\ i\neq j}}^{r} q_i^{\nu_i}.
	\]
	Since~\cref{assumption:fLemma} implies
	\[
		\innerFunc^{\dagger}(b)-\nu_j
		\leq \innerFunc^{\dagger}(b)-
		\innerFunc^{\dagger}(ab)
		= \innerFunc^{\dagger}(b)\parentheses*{
			1
			- \frac{\innerFunc^{\dagger}(ab)}
				   {\innerFunc^{\dagger}(b)}
		}
		\leq - \ConstFLemma \innerFunc^{\dagger}(a) b,
	\]
	which, by assumption, is $\leq - \ConstFLemma q_1^{k}$, we infer
	\[
		\frac{m^*}{m_N}
		= q_j^{\innerFunc^{\dagger}(b)-\nu_j}
		q_{r+1}^{\innerFunc^{\dagger}(a)}
		\leq 
		q_j^{-\ConstFLemma \innerFunc^{\dagger}(a) b} 
		q_{r+1}^{\innerFunc^{\dagger}(a)}.
	\]
	However, this shows that $m^{*}<m_{N}$, which contradicts the
	definition of $m_{N}$, for we have
	\[
		\funcName(m^*)
		= \underbrace{ \innerFunc(
		\innerFunc^{\dagger}(b)) \innerFunc(\innerFunc^{\dagger}(a)) }_{ {}
		= ab = \innerFunc(\nu_{j}) } 
		\prod_{\substack{i=1 \\ i\neq j}}^{r} \innerFunc(\nu_i)
		= \prod_{i\leq r} \innerFunc(\nu_i)
		= N.
	\]
	Hence, we conclude that $\Omega_{Q}(\innerFunc(\nu_{j})) \leq k$.
\end{proof}

\begin{lem}\label{lem:Cauchy-Schwarz}
	Let $\nu_1,\ldots,\nu_t$ be positive integers. Then
	\begin{equation}
		\sum_{j\leq t} \log \innerFunc(\nu_j)
		\leq \frac{C_g}{2} \parentheses[\bigg]{ \sum_{j\leq t}j\nu_j }^{1/2},
		\label{eq:CS:ineq}
	\end{equation}
	where $C_g$ is given by~\cref{eq:def:C}.
	If additionally $\nu_t \geq \nu$, then
	\[
		\sum_{j\leq t}\log \innerFunc(\nu_j)
		\ll \sqrt{\frac{1}{\nu^{2\epsilon}}
			+ \frac{(\log \innerFunc(\nu))^2}{\nu}
		} \parentheses[\bigg]{ \sum_{j\leq t}j\nu_j }^{1/2},
	\]
	with $\epsilon$ from~\cref{assumption:fGrowth}.
\end{lem}
\begin{proof}
	(Compare~\cite[Lemma 3.3]{ButtkewitzElsholtzFord-SP:2012}.)
	First note that the right hand side of~\cref{eq:CS:ineq} 
	is minimal if the $\nu_j$s are decreasing. 
	Hence, we may subsequently assume that
	$\nu_1\geq\nu_2\geq\ldots\geq\nu_t$.
	Let $y_i = \#\{j : \nu_j\geq i\}$ and observe that
	\begin{equation}
		\sum_{j\leq t}j\nu_j
		= \sum_{j\leq t} \sum_{i\leq \nu_j} j
		= \sum_{i=1}^{\infty} \sum_{j\leq y_i} j
		= \frac{1}{2} \sum_{i=1}^{\infty} y_i(y_i+1)
		\geq \frac{1}{2} \sum_{i=1}^{\infty} y_i^2.
		\label{eq: estimate for the sum over capital Yj}
	\end{equation}
	By partial summation,
	\begin{equation}
		\sum_{j\leq t} \log \innerFunc(\nu_j)
		= \sum_{i=1}^{\infty} (y_i-y_{i+1}) \log \innerFunc(i)
		= \sum_{i=1}^{\infty} y_i \log\frac{\innerFunc(i)}{\innerFunc(i-1)}.
		\label{eq: formula for small x}
	\end{equation}
	The first claim now follows by applying the
	Cauchy--Schwarz inequality to the right hand side, and
	taking~\cref{eq: estimate for the sum over capital Yj}
	into account.
	
	Moreover, if $\nu_t\geq \nu$, then $y_1=y_2=\ldots=y_{\nu}$ and 
	\[
		\sum_{i\leq A} y_i \log\frac{\innerFunc(i)}{\innerFunc(i-1)} 
		= y_1 \log \innerFunc(\nu).
	\]
	By splitting up the sum in~\cref{eq: formula for small x}
	into sums over the ranges $i\leq \nu$ and $i>\nu$, and applying the
	Cauchy--Schwarz inequality, we obtain
	\[
		\sum_{j\leq t}\log \innerFunc(\nu_j)
		\leq \parentheses[\bigg]{ \sum_{i=1}^{\infty} y_i^2 }^{1/2}
		\parentheses[\bigg]{
			\frac{(\log \innerFunc(\nu))^{2}}{\nu}
			+ \sum_{i>\nu}
			\parentheses*{\log\frac{\innerFunc(i)}{\innerFunc(i-1)}}^2
		}^{1/2}.
	\]
	By~\cref{assumption:fGrowth} and $\log(1+1/i)<1/i$, the second
	sum is $\ll \nu^{-2\epsilon}$.
	In view of~\cref{eq: estimate for the sum over capital Yj},
	we have established the second claim.
\end{proof}

\begin{lem}\label{lem: upper estimate for maximal order of f}
	For every $\varepsilon>0$, and $s \coloneqq \omega_{Q}(n) \geq 2$, 
	\[
		\funcName(n) \ll \parentheses*{
			\frac{(\ConstF + \varepsilon) \log n}{s\log s}
		}^{s}.
	\]
\end{lem}
\begin{proof}
	See~\cite[Lemma 3.2]{ButtkewitzElsholtzFord-SP:2012} and, 
recalling that there
	$\innerFunc$ is $x\mapsto x+1$,  use~\cref{assumption:fsublinear}
	instead of $x+1\leq 2x$.
\end{proof}

% ~~~~~~~~~~~~~~~~~~~~~~~~~~~~~~~~~~~~~~~~~~~~~~~~~~~~~~~~~~~~~~~~~~~~
\section{Proof of \texorpdfstring{\cref{thm:UpperBound}}{Theorem\autoref{thm:UpperBound}}}
%\label{sec:proof:UpperBound}
Let $n$ be a positive integer such that $\funcName(\funcName(n))>1$ and
$N=\Pi_{Q}(\funcName(n))$. As before, $\funcName(\funcName(n))=\funcName(N)$.

We now write $N$ as a product of powers of elements in $Q$ and
split these into two groups according to the size of their exponents.
More precisely, we write $N=N'N''$, where
\[
	N'  = u_1^{b_1} \cdots u_w^{b_w}, \quad
	N'' = v_1^{a_1} \cdots v_s^{a_s}
\]
and $u_1<\ldots<u_w$, $v_1<\ldots<v_s$ all belong to $Q$,
are all distinct, and $a_{i}\leq(\log_2 n)^K$ and
$b_{i}>(\log_2 n)^{K}$, for $K=\max\{ 6, 2/\epsilon \}$, 
with $\epsilon$ from~\cref{assumption:fGrowth}.

Clearly, $\log \funcName(N) = \log \funcName(N') + \log \funcName(N'')$, 
so that it suffices to deal with $\funcName(N')$ 
and $\funcName(N'')$ separately. The main term in~\cref{eq:thm:UpperBound}
comes from $\log \funcName(N'')$ (see~\cref{eq: upper bound for f of N''}) and the term $\log \funcName(N')$ is seen to be somewhat smaller
(see~\cref{eq: upper bound for f of N'}).

% --------------------------------------------------------------------
\subsection{Bounding \texorpdfstring{$\boldsymbol{\funcName(N')}$}{f(N')}}
Write $m_{N'} = q_1^{\beta_1}\cdots q_h^{\beta_h}$.
Due to \cref{lem:m_N}~\cref{lem:m_N:division} we have
$m_{N'} \leq m_N \leq n$ and, hence, $h\ll\log n$.
\cref{lem:m_N}~\cref{lem:m_N:OmegaEstimate}
yields $\Omega(\innerFunc(\beta_i)) \ll \log_2 h \ll \log_3 n$ for every $i$.
Therefore, there are $\gg b_j/\log_3 n$ values of $i$ such that
$u_j\mid \innerFunc(\beta_i)$. Furthermore, assuming, as we may, that $n$
is sufficiently large, \cref{lem:Cauchy-Schwarz} with
$\nu = \lfloor (\log_2 n)^K \rfloor$ shows that,
for $\epsilon'=K/2-2$,
\begin{align*}
	&
	\log \funcName(N')
	= \sum_{j\leq w} \log \innerFunc(b_j)
	\ll (\log_2 n)^{
		-\min\{ \epsilon K, 2 \}
	} \parentheses[\bigg]{ \sum_{j\leq w} jb_j }^{1/2}.
\end{align*}
Moreover, 
\begin{align*}
	\frac{1}{\log_3 n} \sum_{j\leq w} jb_j
	&\leq \sum_{j\leq w} \frac{u_j b_j}{\log_3 n}
	\\
	&\ll \sum_{i\leq h} \sum_{p\mid \innerFunc(\beta_i)} p
	\leq \sum_{i\leq h} \innerFunc(\beta_i) \\
	&\ll \sum_{i\leq h} \beta_{i}\ll\log m_{N'}
	\leq\log n.
\end{align*}
Hence,
\begin{equation}\label{eq: upper bound for f of N'}
	\log \funcName(N')
	\ll \frac{
		\sqrt{(\log n) \log_3 n}
	}{
		(\log_2 n)^{2}
	}.
\end{equation}
% --------------------------------------------------------------------
\subsection{Bounding \texorpdfstring{$\boldsymbol{\funcName(N'')}$}{f(N'')}}
To estimate $\funcName(N'')$ we may assume that 
\begin{equation}\label{eq:sSize}
	s > \frac{\sqrt{\log n}}{(\log_2 n)^{K/2}},
\end{equation}
for otherwise \cref{lem: upper estimate for maximal order of f}
implies that
\[
	\log \funcName(N'') \ll \frac{\sqrt{\log n}}{(\log_2 n)^{K/2-1}}.
\]
We shall prove the following proposition that is crucial for estimating $f(N')$;
it relates $m_{N''}$ with upper bounds as in \cref{lem:Cauchy-Schwarz}.
\begin{prop} 
	\label{prop: relating weighted sum over exponents with inverse function}
	Let $K=\max\{ 6, 2/\epsilon \}$, 
	with $\epsilon$ from~\cref{assumption:fGrowth}.
	Suppose $N'' = v_1^{a_1} \cdots v_s^{a_s}$
	where $u_1<\ldots<u_w$, $v_1<\ldots<v_s$ all belong to $Q$,
	are all distinct, and $a_{i}\leq(\log_2 n)^K$, 
	and $s$ satisfies~\cref{eq:sSize}. Then, 
	\[
	\log m_{N''} \geq  \parentheses*{
			1 + \LandauO\parentheses*{ \frac{\log_3 n}{\log_2 n} }
		} \ConstGDagger \ConstQAsympt
		\frac{(\log_2 n)^2}{4} \sum_{j\leq s} ja_j.
	\]
\end{prop}
Let us suppose for the moment that  
\cref{prop: relating weighted sum over exponents
with inverse function} is proved. We can conclude 
by \cref{lem:m_N}~\cref{lem:m_N:division} that
\[
	\log n \geq \log m_{N''}
	\geq \parentheses*{
		1 + \LandauO\parentheses*{ \frac{\log_3 n}{\log_2 n} }
	} \ConstGDagger \ConstQAsympt
	\frac{(\log_2 n)^2}{4} \sum_{j\leq s} ja_j.
\]
Inequality~\cref{eq:CS:ineq} implies that
\begin{equation}
	\log \funcName(N'')
	\leq \frac{\sqrt{\log n}}{\log_2 n} \parentheses*{
			\frac{C_g}{\sqrt{\ConstGDagger\ConstQAsympt}}
			+ \LandauO\parentheses*{ \frac{\log_3 n}{\log_2 n} }
		},
	\label{eq: upper bound for f of N''}
\end{equation}
which concludes the proof of \cref{thm:UpperBound}.

\begin{proof}[Proof of 
\cref{prop: relating weighted sum over exponents
with inverse function}]
Write $m_{N''}=q_1^{\alpha_1}\cdots q_r^{\alpha_r}$ 
for the minimal element of $f^{-1}(N'')$, as in \cref{lem:m_N}.
Our first goal is to establish that $r$ cannot be too small.
By \cref{lem:m_N}, and letting $s_0=1$ for the moment,
the last sum in
\begin{equation}\label{eq:OmegaN''}
	\Omega(N'') = \sum_{j\leq s} a_j
	= \sum_{i\leq r} \Omega(\innerFunc(\alpha_i))
\end{equation}
is seen to be
\[
	\sum_{k=1}^{\infty} k \parentheses[\Big]{
		\pi_Q\parentheses[\big]{ q_{r+1}^{1/s_{k-1}} }
		- \pi_Q\parentheses[\big]{ q_{r+1}^{1/s_k} }
	} 
	= r+1 + \sum_{k=1}^{\infty} \pi_{Q} \parentheses[\big]{ q_{r+1}^{1/s_{k}} }
	\eqqcolon r + E.
\]
To handle $E$, we split the term for $k=1$ from the sum
and estimate the rest trivially, thereby obtaining
$	E
	%\ll \pi_Q \parentheses{ q_{r+1}^{1/s_1} }
	%	+ \sum_{2\leq k \leq \text{UPPER BOUND}}
	%		\pi_Q \parentheses{ q_{r+1}^{1/s_2} }
	\ll \pi \parentheses{ q_{r+1}^{1/s_{1}} }.
$
Also, by~\cref{eq:OmegaN''}, $\Omega(N'')\geq r$, so that
\[%\label{eq:OmegaN'':GoodBound}
	\Omega(N'') = r + \LandauO\parentheses[\big]{
		\pi \parentheses[\big]{ q_{r+1}^{1/\ConstFLemma q_1} }
	}.
\]
Hence,
\begin{equation}\label{eq:OmegaN'':LazyBound}
	r\leq\Omega(N'') \leq r+r^{\theta},
\end{equation}
where $\theta\in(1/\ConstFLemma q_1,1)$ is some constant
(recall that by~\cref{assumption:fLemma} this interval is non-empty).
In particular, $r\gg s$ so that by~\cref{eq:sSize}, $r$ must be
large if $n$ is sufficiently large.
The next goal is to determine $\innerFunc(\alpha_i)$ for all $i$ in a
suitable range. To this end, first note that by
\cref{lem:m_N}~\cref{lem:m_N:OmegaEstimate} we find that
$\innerFunc(\alpha_{i})$ is prime for all $i>r^{\theta}$.
Let $\varepsilon = (3K+1)(\log_3 n)/\log_2 n$, and assume that
$n$ is sufficiently large as to ensure that $\varepsilon < 1-\theta$.
By~\cref{eq:sSize},
\begin{equation}\label{eq:r:s:Ineq}
	2 r^{\theta}
	\leq 2(\Omega(N''))^{\theta}
	\leq 2 \parentheses*{ s(\log_2 n)^K }^{\theta}
	\leq  s^{1-\varepsilon}
	\leq \sum_{s-s^{1-\varepsilon}<j<s} a_j,
\end{equation}
for $n$ sufficiently large. Hence,
\begin{equation}\label{eq:ReducedOmegaN'':Estimate}
	\sum_{j\leq s-s^{1-\varepsilon}} a_{j}
	\leq \Omega(N'')-2r^{\theta}
	\leq r-r^{\theta}.
\end{equation}
As explained above, $\innerFunc(\alpha_{i})$ is prime for all $i>r^{\theta}$
and from~\cref{eq:ReducedOmegaN'':Estimate} we know that this surely
is the case for all $i \geq r - \sum_{k\leq j} a_k$, where
$j \leq s - s^{1-\varepsilon}$.
Since, by \cref{lem:m_N}~\cref{lem:m_N:factorisation} the values
$\innerFunc(\alpha_i)$ are decreasing as $i$ increases,
this yields that $\innerFunc(\alpha_i)=v_j$ 
for $r-\sum_{k\leq j} a_k < i \leq r - \sum_{k<j} a_k$. 
By~\cref{eq:OmegaN'':LazyBound} and~\cref{eq:r:s:Ineq},
\begin{align}
	\label{ineq:rMinusSum}
	r - \sum_{k\leq j}a_{k}
	= r - \Omega(N'') + \sum_{k\leq s-j} a_{j+k}
	\geq s - j - r^{\theta}
	\geq \frac{1}{2} s^{1-\varepsilon}.
\end{align}
From~\cref{assumption:gDaggerGrowth} 
and~\cref{assumption:qAsymptotic} we deduce that
\begin{align}
	\innerFunc^{\dagger}(q_j) \geq \ConstGDagger q_j + \LandauO(q_j/\log q_j)
	\geq \ConstGDagger\ConstQAsympt j \log j \label{ineq:gDagger:qj}
\end{align}
for all sufficiently large $j$. Hence, by~\cref{ineq:gDagger:qj} and~\cref{ineq:rMinusSum}, 
\[
	 \log m_{N''}\geq\ConstGDagger\ConstQAsympt \sum_{
		s^{1-\varepsilon} \leq j \leq s-s^{1-\varepsilon}
	} j (\log j) a_j \parentheses*{ \log s + \LandauO(\log_3 n) }.
\]
By~\cref{eq:sSize}, we find that the right hand side above exceeds
\begin{align*}
	\MoveEqLeft
	\sum_{
		s^{1-\varepsilon} \leq j \leq s-s^{1-\varepsilon}
	} (1-\varepsilon) (\log s)^2 j a_j \parentheses*{
		1 + \LandauO\parentheses*{ \frac{\log_3 n}{\log s} }
	}\\
	&\geq \sum_{
		s^{1-\varepsilon}\leq j\le s-s^{1-\varepsilon}
	} \parentheses*{ 1+\LandauO(\varepsilon) } (\log_2 n)^2 j a_j.
\end{align*}
By the choice of $\varepsilon$, 
we get $s^{\varepsilon} \gg (\log_2 n)^{-K-1}$.
Also, $\sum_{j\leq s} ja_j \geq \frac{1}{2} s^2$.
Now recalling that $a_{j}\leq(\log_2 n)^{K}$ for every $j$,
we infer that
\begin{align*}
	\sum_{ s^{1-\varepsilon} \leq j\leq s-s^{1-\varepsilon} } j a_j
	&= \sum_{j\leq s} ja_j + \LandauO\parentheses*{
			s^{2-\varepsilon}(\log_2 n)^K
		} \\ &
	= \parentheses*{
		1 + \LandauO\parentheses*{ \frac{1}{\log_2 n} }
} \sum_{j\leq s} ja_j,
\end{align*}
thus completing the proof.
\end{proof}

% ~~~~~~~~~~~~~~~~~~~~~~~~~~~~~~~~~~~~~~~~~~~~~~~~~~~~~~~~~~~~~~~~~~~~
\section{Proof of \texorpdfstring{\cref{thm:LowerBound}}{Theorem\autoref{thm:LowerBound}}}
%\label{sec:proof:LowerBound}
Recall that the main term in the upper bound
in \cref{thm:UpperBound} stems from an application of
\cref{lem:Cauchy-Schwarz}. Given some large $x>1$, 
we wish to find an integer $n$ smaller than $x$, such that
\[
	\log \funcName(\funcName(n)) 
	= \sum_{\substack{q\in Q \\ q \mid \funcName(n)}}
		\log \innerFunc(\nu_q(\funcName(n)))
\]
is large. The idea is to realise equality in
\cref{lem:Cauchy-Schwarz}. Therefore, recalling that the
inequality was obtained by applying the Cauchy--Schwarz inequality
to~\cref{eq: formula for small x}, we would like to have
\[
	\#\{ q\in Q : \nu_q(\funcName(n)) \geq i \}
	\approx \text{const} \times \log \frac{\innerFunc(i)}{\innerFunc(i-1)}
	\quad (i\geq 1)
\]
with some constant, independent of $i$. 
Furthermore, to have suitable control over $\funcName(n)$ it seems
reasonable to choose $n$ such that the factorisation of
$\funcName(n)$ is known. With this in mind, let
$\varepsilon = \ConstCe\frac{\log_3 x}{\log_2 x}$ 
for $\ConstCe$ sufficiently large, where 
\[
	t = \floor*{
		\parentheses*{ \frac{8\log \innerFunc(1)}{C_g}- \varepsilon}
		\frac{\sqrt{\log x}}{\log_2 x}
	},
\]
and consider
\[
	\nu_j \coloneqq \floor*{
		1-\frac{1}{\alpha}+\frac{1}{(\alpha+1)^{j/t}-1}
	} \quad (1\leq j\leq t).
\]
Evidently,
\begin{equation}
	\nu_j = \frac{1}{\log(\alpha+1)} \frac{t}{j} + \LandauO(1)
	\label{eq:ai:Asymptotics}
\end{equation}
Letting
\[
	n = \prod_{j\leq t} \prod_{i\leq \nu_j}
		q_{ \nu_1 + \ldots + \nu_{j-1} + i }^{ \innerFunc^{\dagger}(q_j) },
\]
we find that
\[
	\funcName(n) = \prod_{j\leq t} 
	\innerFunc(\innerFunc^{\dagger}(q_j))^{\nu_j}
	= \prod_{j\leq t} q_j^{\nu_j}.
\]
Now it remains to give a good lower bound on
$\log \funcName(\funcName(n))$ and an upper bound on $n$. 
To obtain the upper bound, let
\begin{equation}
	y_i = \#\{ j : \nu_j \geq i \}
	= \floor*{ \frac{t}{\log(\alpha+1)}
		\log\parentheses*{ 1 + \frac{1}{i-1+\alpha^{-1}} }
	}.
	\label{eq:Yj:def}
\end{equation}
Observe that $\nu_1+\ldots+\nu_t \ll t\log t$. 
Using~\cref{assumption:qAsymptotic} we find that
\[
	\log q_{\nu_1+\ldots+\nu_t} \leq \log t + 2\log_2 t + \LandauO(1).
\]
Hence,
\begin{align*}
	\log n &\leq \sum_{j\leq t} \nu_j \innerFunc^{\dagger}(q_j)
		\log q_{\nu_{1}+\ldots+\nu_j} \\ &
	\leq \frac{\ConstQAsympt}{\alpha} \parentheses*{
			(\log t)^{2}+3(\log_2 t)\log t + \LandauO(\log t)
		} \sum_{j\leq t} j\nu_j.
\end{align*}
Since $y_i = \LandauO(t/i)$ and by~\cref{eq:ai:Asymptotics}
and~\cref{eq:def:C},
\begin{align*}
	\sum_{j\leq t}j\nu_{j} &
	= \frac{1}{2}\sum_{i\leq \nu_{1}}y_i(y_i+1) \\ &
	= \frac{t^2}{2(\log(\alpha+1))^2}
		\sum_{i=1}^{\infty} \parentheses*{
			\log\parentheses*{ 1+\frac{1}{i-1+\alpha^{-1}} }
		}^{2} + \LandauO(t\log t) \\ &
	=\frac{t^{2}C_g^{2}}{16(\log(\alpha+1))^{2}}+\LandauO(t\log t).
\end{align*}
By the definition of $t$,
$\log t = \frac{1}{2} \log_2 x - \log_3 x + \LandauO(1)$
and $\log_2 t = \log_3 x + \LandauO(1)$. 
By choosing $\ConstCe$ sufficiently large, we get 
\[
	\parentheses*{ 1+\LandauO\parentheses*{ \frac{\log_3 x}{\log_2 x} } }
	\parentheses*{
		1 - \frac{C_g \ConstCe}{8\log(\alpha+1)}
		\frac{\log_3 x}{\log_2 x}
	}^{2}
	\leq 1.
\]
Thus, we infer
\[
	\log n \leq \frac{\ConstQAsympt}{\alpha} \parentheses*{
		1+\LandauO\parentheses*{ \frac{\log_3 x}{\log_2 x} }
	}
	\parentheses*{
		1-\frac{\varepsilon C_g}{8\log(\alpha+1)}
	}^2 \log x
\]
so that $n\leq x^{\ConstQAsympt/\alpha}$ if $x$ is sufficiently
large. Next, we estimate $\log \funcName(\funcName(n))$: 
Using partial summation and~\cref{eq:Yj:def},
\begin{align*}
	\log \funcName(\funcName(n)) &= \sum_{j\leq t}\log \innerFunc(\nu_{j})
	= \sum_{i=1}^{\infty} (y_i-y_{i+1}) \log \innerFunc(i) \\ &
	= \sum_{i=1}^{\infty} y_i \log\frac{\innerFunc(i)}{\innerFunc(i-1)}.
\end{align*}
Due to the construction of $n$ the last sum simplifies to: 
\begin{align*}
	\MoveEqLeft
	\sum_{i\leq \nu_1} y_i \log\frac{\innerFunc(i)}{\innerFunc(i-1)} \\ &
	= \sum_{i\leq \nu_{1}} \parentheses*{
		\frac{t}{\log(\alpha+1)} \parentheses*{
			\log\frac{\innerFunc(i)}{\innerFunc(i-1)}
		}^2 + \LandauO(1/i)
	} \\ &
	= \frac{C_g^2}{8\log(\alpha+1)}t + \LandauO(\log t) \\ &
	= \frac{\sqrt{\log x}}{\log_2 x} \parentheses*{
		C_g + \LandauO\parentheses*{ \frac{\log_3 x}{\log_2 x} }
	}.
\end{align*}
Since $M(x^{\ConstQAsympt/\alpha}) \geq \log \funcName(\funcName(n))$, we
infer~\cref{eq:thm:LowerBound}. This concludes the proof.

% ~~~~~~~~~~~~~~~~~~~~~~~~~~~~~~~~~~~~~~~~~~~~~~~~~~~~~~~~~~~~~~~~~~~~
\section*{Acknowledgements}
Most parts of the present work were completed whilst M.~T.\ and N.~T.\ were working at Würzburg University and Graz University of Technology, respectively. The financial support of both institutions is highly appreciated.
The authors would like to thank Jan-Christoph Schlage-Puchta for comments on an earlier version of this article.
Moreover, the detailed suggestions of the anonymous referee
are gratefully acknowledged. In particular,
the question for further applications turned out to be fruitful. 

% --------------------------------------------------------------------
%\bibliographystyle{plainnat}
%\nocite{*}
%\bibliography{maxord}
% --------------------------------------------------------------------

% ~~~~~~~~~~~~~~~~~~~~~~~~~~~~~~~~~~~~~~~~~~~~~~~~~~~~~~~~~~~~~~~~~~~~
\end{document}